\documentclass[11pt]{article}
\usepackage[utf8]{inputenc}
\usepackage{amsmath,latexsym,amssymb}

\usepackage{multirow}
\usepackage{euscript}
\usepackage{eufrak}
\usepackage{amssymb}
\usepackage{amsmath}
\usepackage{tikz}
\usepackage{mathrsfs}

\newenvironment{axioms}{\begin{description}}{\end{description}}
\newcommand\AxiomItem[2]{\item[\qquad #1] \quad $#2$}

\newenvironment{proof}{\noindent\bf Proof. \rm}{\hfill $\mbox{\boldmath{$ \square$}}$}

\newcommand{\FIXMA}[1]{\verb+#1+}

\newcommand{\lthen}{\mathbin{\rightarrow}}    

\newtheorem{coro}{\bf Corollary}[section]
\newtheorem{theo}[coro]{\bf Theorem}
\newtheorem{defi}[coro]{\bf Definition}
\newtheorem{lem}[coro]{\bf Lemma}
\newtheorem{rem}[coro]{\bf Remark}
\newtheorem{lemma}[coro]{\bf Lemma}

\title{ \large \textbf{ Paraconsistent models of Zermelo-Fraenkel set theory}}

\author{Aldo Figallo-Orellano\footnote{E-mail: \texttt{aldofigallo@gmail.com}}\,\,  and Juan Sebasti\'an Slagter\footnote{E-mail:  \texttt{juan.slagter@uns.edu.ar}}\\ [2mm] %
{\small Departamento de Matem\'atica, Universidad Nacional del Sur (UNS),}\\
{\small  Bah\'ia Blanca, Argentina}}
\date{}

\begin{document}

\maketitle
\begin{abstract}

 
In this paper, we build Fidel-structures valued models following the methodology developed for Heyting-valued models; recall that Fidel structures are not algebras in the universal algebra sense. Taking  models that verify Leibniz law, we are able to prove that all set-theoretic axioms of ZF are valid over these models. The proof is strongly based on the existence of paraconsistent models of Leibniz law. In this setting, the difficulty of having algebraic paraconsistent  models of law for formulas with negation  using the standard interpretation map is discussed, showing that the existence of models of Leibniz law is essential to getting models for ZF\footnote{ It is worth mentioning that the authors presented the content of this paper in a Brazilian congress in 2020, see \cite{FS}.}.   
 
\end{abstract}

\section{Introduction}

Paraconsistency is the study of logic systems having a negation $\neg$  which is not explosive, that is, there exist formulas $\alpha$ and $\beta$ in the language of the logic such that $\beta$ is not derivable from the contradictory set $\{\alpha, \neg\alpha\}$.  These systems are typically called non-trivial systems. There are several approaches to paraconsistency in the literature since the introduction in 1948 of Jaskowski’s system of Discussive logic such as Relevant logics, Adaptive logics, Many-valued logics, and many others. The well-known $3$-valued logic  of Paradox (LP) was introduced by  Priest  with the aim of formalizing the philosophical perspective underlying  Priest and  Sylvan's Dialetheism. As it is well-known, the main thesis behind Dialetheism is that there are true contradictions, that is, that some sentences can be both true and false at the same time and in the same way, see \cite{GP}. 

On the other side of paraconsistency, we have the C-systems $C_n$ and $C_\omega$ that are among the best-known contributions of da Costa, his students and  collaborators. In the early 1960's, these systems were introduced in the da Costa's Habilitation thesis; in particular, for $C_\omega$ da Costa proceeded axiomatically preserving the positive part of intuitionistic logic, and changing the axioms for negation, \cite{dC}. The motivation of inventing these systems was to have  non-explosive systems; as a causal consequence, these systems were not congruential for formulas with negation, so that the systems remained without semantics for several years. Afterward, Fidel presented semantics for  $C_n$ and $C_\omega$ by means of presenting a novel algebraic-relational structures in order to give a Adequacy Theorems for these logics w.r.t. those structures in the   early 1970's, \cite{F1}.  Nowadays, these structure are called Fidel structures and it is important to note  that they are not algebras in the  universal algebra sense. It turned out later that indeed  da Costa's systems are not algebraizable in the Blok-Pigozzi's method, see, for instance, \cite{OFP}. 

Recall that Fidel structures are pairs $\langle {\bf A}, \{N_x\}_{x\in A} \rangle$ where $\bf A$ is a {\em generalized Heyting algebra} and  $N_x$ is a set of all possible negations of $x\in A$. This kind of semantics was presented to Paraconsistent Nelson's logic by Odintsov, \cite[Section 3]{Odintsov}; in this case, the logic is algebraizable, but not congruential for formulas with negation.

Focusing on the primary aim of this note, let us recall that Boolean-valued models of Zermelo-Fraenkel set theory (ZF) were introduced by Scott,  Solovay and  Vop\u{e}nka in 1965;  the development of this theory  can be found in Bell's book, \cite{Bell}. With these models, it is possible to prove  the validity of   set-theoretic axioms of ZF. Now, taking  Heyting algebra instead of Boolean algebra, we can construct Heyting-valued models where the proof the validity of axioms of ZF on these models is obtained from  adapting the Boolean case, see  \cite{Bell1}. Other  set theories can be presented such as quantum and  fuzzy ones using appropriate lattice-valued models.  More recently, L\"owe and Tarafder presented the class of reasonable implication algebras in order to construct algebraic-valued models that validate all axioms of the negation-free fragment of ZF, \cite{BT}; in this paper,  the importance of Leibniz law was settled down in Section 4. In the literaturee, there are some generalization of algebraic-valued models -- see the paper of L\"owe {\em et al.}   \cite{BT1}-- and  several approaches of paraconsistent set theories with models build over algebraic-like models, which  have no relation with our presentation. On the other hand,  Priest and Ferguson presented paraconsistent models for the arithmetic, see \cite{Pri,Fer}.   In 2020, Figallo-Orellano and Slagter presented  models for da Costa's paraconsistent set theory via models built over Fidel structures, \cite{FS}.  


The principal purpose of this note is to show that Zermelo-Fraenkel set theory has paraconsistent models built over Fidel  structures for $C_\omega$.  To this end, we start by building  Fidel-structure valued models  following the methodology developed for Heyting-valued models. Subsequently, we  take  models that verify Leibniz law and later prove that  such models exist, see the important Remark \ref{important}. Studying Leibniz law on the algebraic setting, we observe that  the law is not verified over  a $\mathbf{H}_3^\ast$-valued model, where $\mathbf{H}_3^\ast$ is the three-valued Heyting algebra with dual pseudo-complement; furthermore, it is possible to see that at least one of the axioms of ZF is not valid over this model using the standard interpretation map, see Section \ref{S4.1}. In contrast, taking a model over the saturated   Fidel-structure over three-element chain, we can see that Leibniz law is verified; and what is more important, this model is paraconsistent. Finally, we present a proof  that Zermelo-Fraenkel set-theoretic axioms are valid over special Fidel-structure valued models; specifically, the ones that verify Leibniz law.

\section{Fidel's semantics for da Costa $C_\omega$ logic }\label{preli}

In this section, we will briefly summarize the main background about Fidel's strutures for $C_\omega$ to be used in the rest of the paper.   Let us start considering the  signature $\Sigma=\{\wedge,\vee,\to,\neg\}$, and the language  $\mathfrak{Fm}$, or set of formulas, over the denumerable set of propositional variables $Var$. The logic $C_\omega$ is defined by the following set of axiom schemas and the rule {\em Modus Ponens} (\cite{dC}):

\begin{itemize}

\item[\rm (A1)] $\alpha \to (\beta \to \alpha)$,

\item[\rm (A2)] $(\alpha \to (\beta \to \gamma)) \to ((\alpha \to \beta) \to (\alpha \to \gamma))$,

\item[\rm (A3)] $(\alpha \wedge \beta) \to \alpha$,
\item[\rm (A4)] $(\alpha \wedge \beta) \to \beta$,
\item[\rm (A5)] $\alpha \to (\beta \to(\alpha \wedge \beta))$,
\item[\rm (A6)] $\alpha \to (\alpha \vee \beta)$,
\item[\rm (A7)] $\beta \to (\alpha \vee \beta)$,
\item[\rm (A8)] $(\alpha \to \gamma) \to ((\beta \to \gamma) \to ((\alpha \wedge \beta) \to \gamma))$,
\item[\rm (A9)] $\alpha\lor \lnot \alpha$,
\item[\rm (A10)] $\lnot\lnot \alpha\lthen \alpha$.
\end{itemize}

 The notion  of derivation of a formula $\alpha$ in $C_\omega$ is defined as usual.  We say that $\alpha$ is derivable from $\Gamma$ in $C_\omega$, denoted by $\Gamma\vdash\alpha$, if there exists a derivation of $\alpha$ from $\Gamma$ in $C_\omega$.  If $\Gamma=\emptyset$ we denote $\vdash\alpha$; in this case, we say that $\alpha$ is a theorem of $C_\omega$. 

Now, recall that an algebra   ${\bf A}=\langle A,\vee,\wedge,\to,0,1\rangle$ is said to be a {\em Heyting algebra} if the reduct $\langle A,\vee,\wedge,0,1\rangle$  is a bounded distributive lattice and for any $a,b\in A$ the value of $a\to b $  is a pseudo-complement of $a$ with respect to $b$; i.e.,  the greatest element of the set $\{z\in A: a\wedge c\leq b\}$. As a more general case, we say that ${\bf B}=\langle B,\vee,\wedge,\to,1\rangle$ is a {\em generalized Heyting algebra} if  $\langle B,\vee,\wedge,1\rangle$  is a distributive lattice with the greatest element $1$ and $\to$ is defied as Heyting case. This last class of algebras is called  {\em Relatively pseudo-complemented lattices} in \cite[Chapter IV]{RA} and {\em Implicative Lattices} in \cite{Odintsov}.

\begin{defi}\label{defestruc} A $ C_\omega$-structure is a  system $\langle {\bf A},\{N_x\}_{x\in A}\rangle$ where $\bf A$ is a generalized Heyting algebra and $\{N_x\}_{x\in A}$ is a family of sets of $A$ such that the following conditions hold for every $x\in A$
\begin{itemize}
  \item[\rm (i)] for very $x\in A$ there is $x'\in N_x$ such that  $x\vee x'=1$,
  \item[\rm (ii)] for every $x'\in N_x$ there is $x''\in N_{x'}$ such that $x''\leq x$.
\end{itemize}

\end{defi}

In what follows,   we sometimes write  $\langle {\bf A},N\rangle$ instead of $\langle {\bf A},\{N_x\}_{x\in A}\rangle$. As example of $ C_\omega$-structure, we can take a generalized Heyting algebra $A$ and the set $N_x^s=\{y\in A:x\vee y=1\}$. The structure $\langle {\bf A},\{N_x^s\}_{x\in A}\rangle$ will be said to be a saturated $C_\omega$-structure.

For a given set of formulas $\Gamma$ in $ C_\omega$, we will consider the binary relation between formulas as follows:

\begin{center}
$\alpha\equiv_\Gamma\beta$ iff $\Gamma \vdash(\alpha\to\beta)\wedge(\beta\to\alpha).$ 
\end{center}

Having in mind the positive axioms of $C_\omega$, (A1) to (A8), we have that $\equiv_\Gamma$ is a congruence with respect to the connectives  $\vee$, $\wedge$ and $\to$. With $|\alpha|_\Gamma$ we denote the class of $\alpha$ under $\equiv_\Gamma$ and $\mathcal{L}_\omega$ denotes the set of all classes. We can define the operations $\vee$, $\wedge$ and $\to$ on  $\mathcal{L}_\omega$  as follows: $|\alpha \# \beta|_\Gamma=|\alpha\#\beta|_\Gamma$ with $\#\in \{\wedge,\vee,\to \}$. Thus, it is clear that $\langle \mathcal{L}_\omega,\wedge,\vee,\to,1\rangle$ is a generalized Heyting algebra with the greatest element $1=| \alpha \to \alpha|_\Gamma$. To extend the latter algebra to $ C_\omega$-structure,  let us  define the set $N_{|\alpha|}$ for each formula $\alpha$ as follows: $N_{|\alpha|_\Gamma} = \{|\neg\beta|_\Gamma: \beta \equiv_\Gamma \alpha\}$. It is not hard to see that $\langle\mathcal{L}_\omega, \{N_{|\alpha|_\Gamma}\}_{\alpha\in \mathfrak{Fm}}\rangle$ is a $ C_\omega$-structure that we will call Lindenbaum structure. 

The idea of taking this kind of Lindenbaum structure is presented by Fidel in \cite{F1} and \cite{F2}; and, it was adapted to Paraconsistent Nelson’s Logic by Odintsov, see \cite[Section 3]{Odintsov} 

We will say that a $ C_\omega$-structure $\langle {\bf A},\{N_x\}_{x\in A}\rangle$ is  a substructure of a $ C_\omega$-structure $\langle {\bf B},\{N'_x\}_{x\in B}\rangle$ if (i) $\bf A$ is a subalgebra of  $\bf B$ and (ii) $N_x\subseteq N'_x$ holds for $x\in A$. It is easy to see that every $ C_\omega$-structure $\langle {\bf A},\{N_x\}\rangle$ is a substructure of the saturated one $\langle {\bf A},\{N_x^s\}\rangle$ defined before.

Besides, we say that a function $v: \mathfrak{Fm}\to \langle {\bf A},\{N_x\}_{x\in A}\rangle$ is a $C_\omega$-valuation if the following conditions hold:
\begin{itemize}
  \item[\rm (v1)] $v(\alpha)\in A$ where $\alpha$ is an  atomic formula,
  \item[\rm (v2)] $v(\alpha\#\beta)=v(\alpha)\# v(\beta)$ where $\#\in\{\wedge,\vee,\to\}$,
  \item[\rm (v3)] $v(\neg\alpha)\in N_{v(\alpha)}$ and $v(\neg\neg\alpha)\leq v(\alpha)$. 
\end{itemize}

For us, a formula $\alpha$ will be semantically valid  if for every $ C_\omega$-structure $\langle {\bf A},N\rangle$ and every valuation $v$ on the structure, the condition $v(\alpha)=1$ holds; and in this case, we denote $\vDash\alpha$. Moreover, we write  $\Gamma\vDash\alpha$ if every $ C_\omega$-structure $\langle {\bf A},N\rangle$ and every valuation $v$ on the structure, the condition $v(\beta)=1$ for every $\beta\in\Gamma$ implies $v(\alpha)=1$. Thus, we have the following Theorem which is the strong version given in \cite[Theorem 5]{F1}.

\begin{theo}{\rm \cite[Theorem 6.3]{OFP}}\label{teo4} Let $\Gamma \cup \{\alpha\}$ be a set of formulas of $C_\omega$. Then, $\Gamma\vdash\alpha$  if only if $\Gamma \vDash\alpha$
\end{theo}

In order to see that $C_\omega$ is a paraconsistent logic, we will consider the structure 

$${\cal M}_3=\langle H_3,N_0=\{1\}, N_\frac{1}{2}=\{1\},N_1=\{0,\frac{1}{2},1\} \rangle$$ 

where  $H_3=(\{0,\frac{1}{2}, 1\}, \wedge, \vee, \to, 0,1)$ is a $3$-valued G\"odel algebra (or $3$-valued Heyting algebra) where $\wedge$ and $\vee$ are infimum and supremum on the three-element chain and $\to$ is given through the Table \ref{table:conecctivesP1}. Now, let us consider a valuation $v$ such that $v(\alpha)= \frac{1}{2}$, $v(\neg \alpha)=1$ and $v(\beta)=0$. So, from Theorem \ref{teo4}, we have $\not\vdash (\neg \alpha \wedge \alpha)\to \beta$ and taking into account the (meta-)deduction theorem, we have $\{\neg \alpha,  \alpha\} \not\vdash \beta$.

\begin{table}[t]
\centering

\begin{tabular}{|c||ccc|}
\hline
$\rightarrow$ & \bf0 & $\bf \frac{1}{2}$ & \bf1 \\ 
\hline
\hline
\bf0  & $1$ & $1$ & $1$  \\
$\bf \frac{1}{2}$  & $0$ & $1$ & $1$ \\
\bf1  & $0$ & $\frac{1}{2}$ & $1$ \\ \hline
\end{tabular}
\caption{Table of $\to$ for  $3$-valued G\"odel algebra}
\label{table:conecctivesP1}
\end{table}

In the paper \cite{OFP}, non-algebraizable extensions of $C_\omega$ were displayed through adding the following axioms:

\begin{itemize}
  \item [(G$_n$)] $ (\alpha_1\to \alpha_2)\vee \cdots \vee (\alpha_{n-2}\to \alpha_{n-1})$,
  \item [(L)] $(\beta \to \alpha)\vee (\alpha\to \beta) $.
\end{itemize}

The logics $C_{\omega,n}$ and $C_{\omega,\infty}$ are obtained from the axiomatic for $C_\omega$ plus (G$_n$) and (L), respectively. Using Fidel-structures for  $C_{\omega,n}$ and $C_{\omega,\infty}$, it was determined that they are decidable systems and verify the following Adequacy Theorem:

\begin{theo}{\rm \cite[Theorem 6.4 and 6.5]{OFP}} Let $\Gamma \cup \{ \alpha\}$ be a set of formulas of $C_{\omega,n}$ ($C_{\omega,\infty}$). Then, $\Gamma\vdash_{C_{\omega,\infty}(C_{\omega,\infty})}\alpha$\,\, if only if\,\, $\Gamma \vDash_{\langle {\bf A},N\rangle}\alpha$\,\, for every $C_{\omega}$-structure $\langle {\bf A},N\rangle$ for  $C_{\omega,n}$ ($C_{\omega,\infty}$).
\end{theo}

\section{  Fidel-structure valued  models}

In this section, we will present Fidel-structure valued  models  as adaptation of Boolean-valued models, \cite{Bell}. To this end, we start by considering the $C_{\omega}$-structure $\langle {\bf A},N\rangle$ where $\bf A$ is a complete lattice; and in this case, we say that $\langle {\bf A},N\rangle$ is a complete $C_{\omega}$-structure.

We fix a model of set theory $\mathbf{V}$ and a completed $C_{\omega}$-structure $\langle {\bf A},N\rangle$. Let us construct a universe of {\em names} by transfinite recursion: 
 
$$ {\mathbf{V}_\xi}^{\langle {\bf A}, N\rangle}=\{x: x\, \textrm{\rm  a function and }\, ran(x)\subseteq A \, \,\textrm{\rm and }\,  dom(x)\subseteq \mathbf{V}_\zeta^{\langle {\bf A}, N\rangle}\, \textrm{\rm for some}\,  \zeta< \xi \} $$

$$ {\mathbf{V}}^{\langle {\bf A}, N\rangle}=\{x: x\in {\mathbf{V}_\xi}^{\langle {\bf A}, N\rangle}\, \textrm{\rm  for some } \xi \} $$

The class ${\mathbf{V}}^{\langle {\bf A}, N\rangle}$ is called the $C_{\omega}$-structure valued model over $\langle {\bf A},N\rangle$. Let us observe that we only need the set $A$ in order to define  ${\mathbf{V}_\xi}^{\langle {\bf A}, N\rangle}$.  By ${\cal L}_\in$, we denote the first-order language of set theory which  consists of only the propositional connectives $\{\to, \wedge,\vee, \neg\}$ of the $C_\omega$ and two binary predicates $\in$ and $\approx$. We expand this language by adding all  the elements of ${\mathbf{V}}^{\langle {\bf A}, N\rangle}$; the expanded language will denoted ${\cal L}_{\langle A, N\rangle}$.

{\bf Induction principles}.  The sets 

$$\mathbf{V}_\zeta=\{x: x\subseteq \mathbf{V}_\xi, \, \textrm{ for some}\, \xi<\zeta\}$$

are definable for every ordinal $\xi$ and then, every set $x$ belongs to $\mathbf{V}_\alpha$ for some $\alpha$. So, this fact induces a function $rank(x)=$ least ordinal $\xi$ such that $x\in \mathbf{V}_\xi$. Since $rank(x)<rank(y)$ is well-founded we induce a {\em principle of induction on rank}: let $\Psi$ be a property over sets. Assume, for every set $x$ if $\Psi(y)$ holds for every $y$ such that $rank(y)<rank(x)$, then $\Psi(x)$ holds. Thus, $\Psi(x)$ for every $x$. From the latter, Induction Principles (IP) holds in ${\mathbf{V}}^{\langle {\bf A}, N\rangle}$. Assume  the following or every $x\in {\mathbf{V}}^{\langle {\bf A}, N\rangle}$: if $\Psi(y)$ holds for every $y\in dom(x)$, then $\Psi(x)$ holds. Hence, $\Psi(x)$ holds for every $x\in {\mathbf{V}}^{\langle {\bf A}, N\rangle}$.   By simplicity, we note every set $u\in {\mathbf{V}}^{\langle {\bf A}, N\rangle}$ by its name $u$ of ${\cal L}_{\langle {\bf A}, N\rangle}$. Besides, we will write $\varphi(u)$ instead of $\varphi(x/u)$. Now, we are going to define a valuation by induction on the complexity of a closed formula in ${\cal L}_{\langle {\bf A}, N\rangle}$ as follows:

\begin{defi}\label{str}

 For a given complete $C_\omega$-structure $\langle {\bf A}, N\rangle$,  the mapping $||\cdot||:{\cal L}_{\langle {\bf A}, N\rangle}\to \langle {\bf A}, N\rangle $ is defined as follows:

\begin{center}

$||u\in v ||^{\langle {\bf A}, N\rangle} =\bigvee\limits_{x\in dom(v)} (v(x) \wedge ||x \approx u ||^{\langle {\bf A}, N\rangle})  $; \\ [3mm]

$||u\approx v ||^{\langle {\bf A}, N\rangle} =\bigwedge\limits_{x\in dom(u)} (u(x)) \to ||x\in v ||^{\langle {\bf A}, N\rangle}) \wedge \bigwedge\limits_{x\in dom(v)} (v(x) \to ||x\in u ||^{\langle {\bf A}, N\rangle})$;\\ [3mm]

$||\varphi \# \psi||^{\langle {\bf A}, N\rangle}  = ||\varphi||^{\langle {\bf A}, N\rangle}  \tilde{\#} ||\psi||^{\langle {\bf A}, N\rangle} $, for every $\#\in \{\wedge,\vee, \to\}$;\\[3mm]

$||\neg \alpha||^{\langle {\bf A}, N\rangle}\in N_{||\alpha||^{\langle {\bf A}, N\rangle}}$ and $||\neg\neg\alpha||^{\langle {\bf A}, N\rangle}\leq ||\alpha||^{\langle {\bf A}, N\rangle}$; \\ [2mm]

$||\exists x\varphi||^{\langle {\bf A}, N\rangle} = \bigvee\limits_{{u\in \mathbf{V}}^{\langle {\bf A}, N\rangle}} ||\varphi (u)||^{\langle {\bf A}, N\rangle}$ and $||\forall x\varphi||^{\langle {\bf A}, N\rangle} = \bigwedge\limits_{{u\in \mathbf{V}}^{\langle {\bf A}, N\rangle}} ||\varphi (u)||^{\langle {\bf A}, N\rangle}$.

\

$||\varphi||^{\langle {\bf A}, N\rangle}$ is called the {\bf truth-value} of the sentence $\varphi$ in the language ${\cal L}_{\langle {\bf A}, N\rangle}$ in the $C_\omega$-structure-valued model over  $\langle {\bf A}, N\rangle$.

\end{center}
\end{defi}

\begin{defi}
A sentence $\varphi$ in the language ${\cal L}_{\langle {\bf A}, N\rangle}$  is said to be valid in ${\mathbf{V}}^{\langle {\bf A}, N\rangle}$, which  is denoted by ${\mathbf{V}}^{\langle {\bf A}, N\rangle} \vDash \varphi$, if $||\varphi||^{\langle {\bf A}, N\rangle}=1$.

\end{defi}

In is important to note that for every completed $C_\omega$-structure $\langle {\bf A}, N\rangle$, the element $\bigwedge\limits_{x\in A} x$ is the first element of $A$ and so, $\bf A$ is a complete Heyting algebra; we denote this element by ''$0$''. 

  Now, we  present a lemma that will be useful in the following:

\begin{lemma}\label{tecnico1}
For a given completed $C_\omega$-structure $\langle {\bf A}, N\rangle$. Then:
\begin{itemize}
\item[\rm (i)]   $|| u \approx u ||^{\langle {\bf A}, N\rangle}=1$;
\item[\rm (ii)] $u(x)\leq || x\in u||^{\langle {\bf A}, N\rangle}$ for every $x\in dom(u)$; 
\item[\rm (iii)] $||u\approx v||^{\langle {\bf A}, N\rangle}=||v\approx u||^{\langle {\bf A}, N\rangle}$, for every $u,v\in {\mathbf{V}}^{\langle {\bf A}, N\rangle}$;
\item[\rm (iv)] $||u\approx v||^{\langle {\bf A}, N\rangle}\wedge ||v\approx w||^{\langle {\bf A}, N\rangle}\leq ||u\approx w||^{\langle {\bf A}, N\rangle} $;
\item[\rm (iv)] $||u\approx v||^{\langle {\bf A}, N\rangle}\wedge ||u \in w||^{\langle {\bf A}, N\rangle}\leq ||v\in w||^{\langle {\bf A}, N\rangle} $.
\end{itemize}
\end{lemma}
\begin{proof}
It has the same proof as that of the intuitionistic case using Leibniz law and the fact that for every closed formula $\phi$ of ${\cal L}_{\langle {\bf A}, N\rangle}$, we have $||\phi||^{\langle {\bf A}, N\rangle}\in {\bf A}$, see for instance \cite{Bell1}.
\end{proof}

\section{Leibniz’s law and its models}

In the classical and intuitionistic set theory, we have that the names represent  objects and if we have equivalent objects, they would have to possess the same properties. This is known as {\em indiscernibility of identicals} and it could be considered as Leibniz law by the following axiom:
 $$ u \approx v   \wedge \varphi(u) \to \varphi(v) $$

In the next, we will take  $\langle {\bf A}, N\rangle$ a complete $C_\omega$-structures such that  $\mathbf{V}^{\langle {\bf A}, N\rangle}$  satisfies Leibniz law; in this case, we will call them {\em Leibniz model}. 

Taking in mind the problem of seeing if such models really exit, we can take complete $C_\omega$-structures ${\cal B}=\langle {\bf B}, N\rangle$ where ${\bf B}$ is a complete Boolean algebra and every set a following singleton $N_x=\{\sim x\}$, for every $x\in B$, being $\sim$ the Boolean negation. Then, it is clear that every ${\cal B}$ validates every axiom of  $C_\omega$; moreover, $\mathbf{V}^{\cal B}$ validates every axiom of ZF but $\mathbf{V}^{\cal B}$ is a consistent model. In the following, we will analyze this matter.

\subsection{Algebraic case}\label{S4.1}

It is worth mentioning that  Leibniz law does not has  algebraic models for formulas with negation. To see this,  we start by considering the Heyting algebra with dual pseudocomplement following   $\mathbf{H}_3^\ast=\langle \{0,\frac{1}{2}, 1\},\wedge,\vee,\to, \neg, 1\rangle$  where the $\to$ and $\neg$ through the Table \ref{tableG3'}, see \cite{Sanka}. Now, let us consider  ${\mathbf{V}}^{\mathbf{H}_3^\ast}$  Dual-Heyting-valued model over $\mathbf{H}_3^\ast$. Let $w\in {\mathbf{V}}^{\mathbf{H}_3^\ast}$  and let $\psi(x)$ be the formula ``$w\in x$'' where $x$ is a free variable. Taking  $v=\{\langle w, 1\rangle\}$, $u=\{\langle w, \frac{1}{2}\rangle\}$, we have that  $||u\approx v||^{\mathbf{H}_3^\ast}= \frac{1}{2}$ and so $|| u \approx v ||^{\mathbf{H}_3^\ast}\leq || \psi(u) \to \psi(v)||^{\mathbf{H}_3^\ast} $ holds. But, since  $|| \psi(v)||^{\mathbf{H}_3^\ast}=1$ and $|| \psi(u)||^{\mathbf{H}_3^\ast}=\frac{1}{2}$, we have that  $||\neg \psi(u)||^{\mathbf{H}_3^\ast}=1$ and $||\neg \psi(v)||^{\mathbf{H}_3^\ast}=0$, and therefore the law is not verified. 

\begin{table}[t]
\centering
\begin{tabular}{|c||ccc|}
\hline
$\rightarrow$ & \bf0 &  $\bf \frac{1}{2}$ & \bf1 \\ 
\hline
\hline
\bf0  & $1$ & $1$ & $1$  \\
$\bf \frac{1}{2}$  & $0$ & $1$ & $1$ \\
\bf1  & $0$ & $\frac{1}{2}$ & $1$ \\ \hline
\end{tabular}
\qquad
\begin{tabular}{|c||c|}
\hline
 & $\neg$ \\ 
 \hline
 \hline
\bf0 & $1$ \\
$\bf \frac{1}{2}$  & $1$ \\
\bf1 & $0$ \\ \hline
\end{tabular}
\caption{Tables of connectives $\rightarrow$ and $\neg$ in $\mathbf{H}_3^\ast$}
\label{tableG3'}
\end{table}

The problem of the non-existence of algebraic models of the law leads  to the impossibility  to prove the validity of at least one axiom of ZF for formulas with negation; indeed, if we take the formula $\phi(x):=\neg \psi(x)$, then it is not hard to see that the axiom ``Separation'' is not valid. 

It worth mentioning that  it is possible to see that there is a logic associated with the matrix $\langle \mathbf{H}_3^\ast, \{1\}\rangle$ which is an extension of $C_\omega$, \cite[Section 3]{O2008}; i.e., this matrix validates all axiom of $C_\omega$.

\subsection{$C_\omega$-structure based on three-element-chain Heyting algebra}\label{rem}

In order to see that there exist paraconsistent  models of Leibniz law, we can consider structures with the following requirements:  $1\in N_x$ for all $x\in A$  and $ N_1=A$. Now, let us consider the $C_{\omega}$-structure-valued model $\mathbf{V}^{\langle {\bf A}, N\rangle}$  over $\langle {\bf A},N\rangle$ such that, for every closed formula $\psi$, we define $||\neg \psi ||^{\langle {\bf A}, N\rangle} =1$ and    $||\neg \neg \psi ||^{\langle {\bf A}, N\rangle} =|| \psi ||^{\langle {\bf A}, N\rangle}$ where the formula $\psi$ is not the form $\neg \phi$ for any closed formula  $\phi$. To see that this condition works well for formulas with more negation, just to bear in mind  that $||\neg \neg \neg \psi ||^{\langle {\bf A}, N\rangle} =||\neg \psi ||^{\langle {\bf A}, N\rangle}=1$ and  $||\neg \neg \neg \neg \psi ||^{\langle {\bf A}, N\rangle} =||\neg \neg \psi ||^{\langle {\bf A}, N\rangle}=|| \psi ||^{\langle {\bf A}, N\rangle}$ hold. In general, we can consider the formula $\neg_e \psi$, where $\neg \cdots \neg \psi$ has an even number  of  negations in front of $\psi$, and $\psi$ is not the form $\neg \phi$; then $||\neg_e \psi ||^{\langle A, N\rangle} =|| \psi ||^{\langle {\bf A}, N\rangle}$. But if the formula $\neg_o \psi$, where $\neg \cdots \neg \psi$ has an odd number of  negations in front of $\psi$, and $\psi$ is not the form $\neg \phi$; then $||\neg_o \psi ||^{\langle {\bf A}, N\rangle} =1$.  In this case, we will say  $\mathbf{V}^{\langle {\bf A}, N\rangle}$ is a standard Leibniz model over a complete saturated $\langle {\bf A},N\rangle$.

For instance, let us again consider the following saturated $C_\omega$-structure (see Section \ref{preli}):

$${\cal M}_3=\langle \{0,\frac{1}{2}, 1\},\wedge,\vee,\to, \neg , N_0, N_\frac{1}{2}, N_1\rangle$$ 

where $N_0=N_\frac{1}{2}=\{1\}$, and $N_1=\{0,\frac{1}{2}, 1\}$ and $\to$ is given in Table \ref{tableG3'}. Now, let's now consider the standard Leibniz model ${\mathbf{V}}^{{\cal M}_3}$ and   $w\in {\mathbf{V}}^{{\cal M}_3}$,  and let $\psi(x)$ be the formula ``$w\in x$'' where $x$ is a free variable. Besides, let us consider the set $v=\{\langle w, 1\rangle\}$, $u=\{\langle w, \frac{1}{2}\rangle\}$.  Then, we have  $||\neg \psi(v)||^{{\cal M}_3}=1=||\neg \psi(u)||^{{\cal M}_3}$ and clearly, $||\neg \neg \psi(u)||^{{\cal M}_3}=||\psi(u)||^{{\cal M}_3}$ and $||\neg \neg \psi(u)||^{{\cal M}_3}=||\psi(u)||^{{\cal M}_3}$. Thus,  the  model ${\mathbf{V}}^{{\cal M}_3}$ verifies  Leibniz law. 

It is clear that we could have taken  $||\neg \psi(v)||^{{\cal M}_3}=\frac{1}{2}$ because the pair $(1,\frac{1}{2})$ is one that verifies Leibniz law. In general, taking a model that verify the law, we do not know which is the real value of  the valuation of any formula with negation  but we will known that it is the right one. Another interesting aspect of   ${\mathbf{V}}^{{\cal M}_3}$ is that it is a paraconsistent model as we have seen in Section \ref{preli}. Moreover,  ${\mathbf{V}}^{{\cal M}_3}$  is a model for the logic $C_{\omega,3}$, and later it will be shown that it is  model of ZF. Now, we are in conditions of offering a proof of the following Theorem.

\begin{theo}\label{LL}
Let $\langle {\bf A}, N \rangle$ be a complete saturated $C_\omega$-structure such that  $\mathbf{V}^{\langle {\bf A}, N\rangle}$ is a standard Leibniz model, then $\mathbf{V}^{\langle {\bf A}, N\rangle}$ verifies the Leibniz law.
\end{theo}
\begin{proof} Let $\varphi(x)$ be a formula of ${\cal L}_{\langle {\bf A}, N\rangle}$ and $u,v\in \mathbf{V}^{\langle {\bf A}, N \rangle}$.  We will make the proof by  induction on the structure of formula. Taking in mind the intuitionistic case, it is possible to see that for the negation-free formula $\varphi(x)$, we have  $||u \approx v||^{\langle {\bf A}, N \rangle}\leq  ||\varphi(u) \to \varphi(v)||^{\langle {\bf A}, N \rangle}$. Now, let us suppose that $\varphi(x)$ is $\neg \psi(x)$ such that $\psi(x)$ is not the form $\neg \phi(x)$ where $\phi(x)$ is any formula, then $||\neg \psi(u)||^{\langle {\bf A}, N \rangle}=1=||\neg \psi(v)||^{\langle A, N \rangle}$. Thus, it is clear that $||u \approx v||\leq  ||\varphi(u) \to \varphi(v)||^{\langle {\bf A}, N \rangle}$. Let us now suppose that  $\varphi(x)$ is $\neg \neg \psi(x)$ such that $\psi(x)$ is not the form $\neg \phi$, then $||\neg \neg  \psi(u)||^{\langle {\bf A}, N \rangle}= || \psi(u)||^{\langle {\bf A}, N \rangle}$ and $||\neg \neg \psi(v)||^{\langle {\bf A}, N \rangle}=|| \psi(v)||^{\langle {\bf A}, N \rangle}$. By inductive hypothesis, we have $\varphi(x)$ verifies the Leibniz law. Now, if $\varphi(x)$ is a formula such that it is the form $\neg \neg \cdots \neg \psi(x)$, we can proceed as the two last cases in view of the remarks made at the beginning of this sub-section. 
\end{proof}

 \

In the following, we  present an important remark that will show  us how work the models of this section.

\begin{rem}\label{important}
 \
 
\noindent Theorem \ref{LL} evidences that there is a way to provide paraconsistent  models that verify Leibniz law, but there could exist others  depending on the choice of the complete saturated $C_\omega$-structure. Furthermore, it is important to note that the sentences $\sim\exists x(x\in x)$ and $\sim\exists xy(x\not\approx y)$ hold in the models of this Section \ref{rem}. This is an unexpected consequence, but we have to take into account that the negation is  paraconsistent. So, we think that this is not a problem on the paraconsistent setting. Recall that this section is presented to see that there exist paraconsistent models of Leibniz law, but taking a Fidel structure based on a  suitable Heyting algebra may have better options. We think it is worth exploring these suitable kind of models of the law.
\end{rem}

We will adopt the following notation, for every formula $\varphi(x)$ and  every $u\in \mathbf{V}^{\langle A, N\rangle}$:  $\exists x\in u \varphi(x)= \exists x(x\in u \wedge \varphi(x))$ and $\forall x\in u \varphi(x)= \forall x(x\in u \to \varphi(x))$.

Now, we  present the following Lemma that will use in the rest of paper.

\begin{lemma}\label{BQ}
Let $\langle {\bf A}, N\rangle$ be a   complete $C_\omega$-structure such that $\mathbf{V}^{\langle {\bf A}, N\rangle}$ is Leibniz model. Then, for every formula $\varphi(x)$ and every $u\in \mathbf{V}^{\langle {\bf A}, N\rangle}$ we have: 
 
 $$|| \exists x\in u \varphi(x)||^{\langle A, N \rangle}= \bigvee\limits_{x\in dom(u)} (u(x) \wedge || \varphi(x)||^{\langle A, N \rangle})$$ 
 
 and 
 $$|| \forall x\in u \varphi(x)||^{\langle A, N \rangle}= \bigwedge\limits_{x\in dom(u)} (u(x) \to || \varphi(x)||^{\langle A, N \rangle}).$$
\end{lemma}

\begin{proof}
It  has the same proof that the intuitionistic case using   Leibniz law, Lemma \ref{tecnico1}. and recalling that for every sentence $\varphi$ and every complete $C_\omega$-structure $\langle {\bf A}, N\rangle$,  the interpretation $|| \varphi||^{\langle A, N \rangle}$ belongs to the Heyting algebra ${\bf A}$, see for instance \cite{Bell1}.
\end{proof}

\section{Zermelo-Fraenkel Set Theory}

The basic system of paraconsistent set theory  is called $ZFC_\omega$ and consists of first order  version of $C_\omega$ over the first-order signature ${\Theta}_\omega$ which contains an equality predicate\, $\approx$ \, and a binary predicate $\in$.

\begin{defi}\label{IZF}

The system $ZFC_\omega$ is the first order theory  obtained from the  logic $C_\omega$     over $\Theta_\omega$  by adding the following set-theoretic axiom schemas:

 \begin{axioms}

\AxiomItem{(Extensionality)}{\forall x \forall y [\forall z(z\in x \leftrightarrow z\in y)\to (z \approx y)]},
\AxiomItem{(Pairing)}{\forall x \forall y \exists w\forall z [z\in w \leftrightarrow (z \approx x \vee z \approx y) ]},
\AxiomItem{(Colletion)}{\forall x [(\forall y\in x\exists z\phi(y,z))\to \exists w\forall y\in x \exists z\in w  \phi(y,z)]},
\AxiomItem{(Powerset)} {\forall x \exists w \forall z [z\in w \leftrightarrow \forall y\in z(y\in x)]},
\AxiomItem{(Separation)} {\forall x \exists w \forall z[z\in w \leftrightarrow (z\in x \wedge \phi(z))]},
\AxiomItem{(Empty set)} {\exists x\forall z[z\in x \leftrightarrow \neg (z\approx z)]},

The set satisfying this axiom is, by extensionality, unique and we refer to it with notation $\emptyset$.

\AxiomItem{(Union)}{\forall x \exists w \forall z[ z\in w \leftrightarrow  \exists y\in x(z\in y)]},

\AxiomItem{(Infinity)}{\exists x [\emptyset\in x \wedge \forall y\in x (y^+\in x)]},

From union, pairing and extensionality, we can note by $y^+$ the unique set $y\cup \{y\}$. 

\AxiomItem{(Induction)}{\forall x [( \forall y\in x \phi(y))\to \phi(x)]\to \forall x\phi(x)}.
\end{axioms}

\end{defi}
\

The last nine axioms are usually used to define the  Zermelo-Fraenkel Set Theory, \cite{Bell}.  Now, we will show some technical results  to be used in the rest of paper.

\begin{defi}
Let $\langle {\bf A}, N\rangle$ be  a complete $C_\omega$-substructure. Given collection of sets $\{u_i: i\in I\} \subseteq \mathbf{V}^{\langle {\bf A}, N\rangle}$ and $\{a_i: i\in I \}\subseteq A$, the mixture $\Sigma_{i\in I} a_i\cdot u_i$ is the fucntion $u$ with $dom(u)=\bigcup\limits_{i\in I} dom(u_i)$ and, for $x\in dom(u)$,  $u(x)= \bigvee\limits_{i\in I} a_i \wedge || x\in u_i||^{\langle {\bf A}, N\rangle}$.
\end{defi}

The following result is known as {\em Mixing Lemma} and its proof is exactly the same for intuitionistic case because it is an assertion about positive formulas, see for instance, \cite{Bell,Bell1}. 

\begin{lem}
Let $\langle {\bf A}, N\rangle$ be  a complete $C_\omega$-substructure and let $u$ be the mixture $\Sigma_{i\in I} a_i\cdot u_i$ where $\{u_i: i\in I\} \subseteq \mathbf{V}^{\langle {\bf A}, N\rangle}$ and $\{a_i: i\in I \}\subseteq A$. If $a_i\wedge a_j\leq || u_i \approx u_j||^{\langle {\bf A}, N\rangle}$ for all $i,j\in I$, then $a_i\leq || u_i \approx u||^{\langle {\bf A}, N\rangle}$.  
\end{lem}

 A set $B$ refines a set $A$ if for all $b\in B$ there is some $a\in A$ such that $b\leq a$. A Heyting algebra $\bf H$ is refinable if for every subset $A\subseteq H$ there exists some anti-chain  $B$ in $H$ that refines $A$ and verifies $\bigvee A =\bigvee B$.

\begin{theo}
Let $\langle {\bf A}, N\rangle$ be a complete $C_\omega$-substructure such that $A$ is  refinable. If $\mathbf{V}^{\langle {\bf A}, N\rangle}\vDash \exists x \psi(x)$, then there is $u\in \mathbf{V}^{\langle {\bf A}, N\rangle}$ such that $\mathbf{V}^{\langle {\bf A}, N\rangle}\vDash \psi(u)$. 
\end{theo} 

\begin{proof} Since $A$ is a set, then the collection $X=\{|| \psi(u)||^{\langle {\bf A}, N\rangle}:u\in \mathbf{V}^{\langle {\bf A}, N\rangle} \}$ is also a set. By the Axiom of Choice, there is an ordinal $\alpha$ and a set $\{ u_\zeta : \zeta < \alpha\} \subseteq \mathbf{V}^{\langle {\bf A}, N\rangle} $ such that $X=\{|| \psi(u_\zeta)||^{\langle {\bf A}, N\rangle}:\zeta < \alpha \}$. Hence, $||\exists x \psi(x)||^{\langle {\bf A}, N\rangle}=\bigvee\limits_{\zeta < \alpha} || \psi(u_\zeta)||^{\langle {\bf A}, N\rangle}$. Now let $\{a_i:i\in I\}$ be a refinement of the set $\{|| \psi(u_\zeta)||^{\langle {\bf A}, N\rangle}:\zeta < \alpha \}$ such that $\bigvee\limits_{i\in I} a_i = \bigvee\limits_{\zeta < \alpha }|| \psi(u_\zeta)||^{\langle {\bf A}, N\rangle}$. We can choose for every $a_i$ some $v_i\in  \mathbf{V}^{\langle {\bf A}, N\rangle}$ such that $a_i\leq || \psi(v_i)||^{\langle {\bf A}, N\rangle}$. Now, let $u$ be the mixture $\Sigma_{i\in I} a_i\cdot v_i$. By the Mixing Lemma and Leibniz law,  we have $1=\bigvee \limits_{i\in I} a_i =\bigvee  \limits_{i\in I} \big(a_i \wedge || v_i=u||^{\langle {\bf A}, N\rangle}\big) \leq \bigvee\limits_{i\in I} \big(|| \psi(v_i)||^{\langle {\bf A}, N\rangle} \wedge || v_i= u||^{\langle {\bf A}, N\rangle}\big)\leq || \psi(u)||^{\langle {\bf A}, N\rangle}$ which completes the proof.
\end{proof}

\

Now, given a complete $C_\omega$-substructure $\langle {\bf A}', N'\rangle$ of $\langle {\bf A}, N\rangle$, we have the associated models $\mathbf{V}^{\langle {\bf A}', N'\rangle}$ and $\mathbf{V}^{\langle {\bf A}, N\rangle}$. Then, it is easy to see that $\mathbf{V}^{\langle {\bf A}', N'\rangle}\subseteq \mathbf{V}^{\langle {\bf A}, N\rangle}$.

On the other hand, we say that a formula $\psi$ is {\em restricted} if all quantifiers are of the form $\exists y\in x$ or $\forall y\in x$, then we have:  

\begin{lem}
For any complete $C_\omega$-substructure $\langle {\bf A}', N'\rangle$ of $\langle {\bf A}, N\rangle$ and any restricted negation-free formula $\psi(x_1,\cdots,x_n)$ with variables in  $\mathbf{V}^{\langle {\bf A}', N'\rangle}$, the following equality holds: 
$$||\psi(x_1,\cdots,x_n)||^{\langle {\bf A}', N'\rangle}=||\psi(x_1,\cdots,x_n)||^{\langle {\bf A}, N\rangle}.$$
\end{lem}

\begin{proof} It is immediate from intuitionistic case, taking into account that  ${\bf A}'$ and ${\bf A}$ are complete Heyting algebras, and $\psi(x_1,\cdots,x_n)$ is a restricted negation-free formula.
\end{proof}

\

Let us now consider the Boolean algebra ${\bf 2}=\langle\{0,1\},\wedge,\vee,\neg,0,1\rangle$ and the natural mapping $\hat{\cdot }:  \mathbf{V}^{\langle {\bf A}, N\rangle} \to \mathbf{V}^{\langle {\bf 2} , N_{\bf 2}\rangle}$ where $N_{\bf 2}=\{(0,1),(1,0)\}$ defined by $\hat{u}=\{\langle \hat{v}, 1\rangle: v\in u\}$. This is well defined by recursion on $v\in dom(u)$. It is clear that $\langle {\bf 2} , N_{\bf 2}\rangle$ is $C_\omega$-substructure of any $\langle {\bf A}, N\rangle$, then the following lemma holds:

\begin{lem}\label{tecnico}
\begin{itemize}
\item[\rm (i)] $|| u\in \hat{v}||=\bigvee\limits_{x\in v} || u\approx \hat{x} ||$ for all $v\in \mathbf{V}$ and $u\in \mathbf{V}^{\langle {\bf A}, N\rangle}$,

\item[\rm (ii)] $u\in v \leftrightarrow \mathbf{V}^{\langle {\bf A}, N\rangle} \vDash \hat{u}\in \hat{v}$ and $u=v \leftrightarrow \mathbf{V}^{\langle {\bf A}, N\rangle} \vDash \hat{u} \approx \hat{v}$,

\item[\rm (iii)] for all $x\in \mathbf{V}^{\langle {\bf 2} , N_{\bf 2}\rangle}$ there exists a unique $v\in \mathbf{V}$ such that $\mathbf{V}^{\langle {\bf 2}, N_{\bf 2}\rangle}\vDash  x\approx \hat{v}$,

\item[\rm (iv)] for any negation-free formula $\psi(x_1,\cdots,x_n)$  and any $x_1, \cdots,x_n\in \mathbf{V}$, we have  $\psi(x_1,\cdots,x_n) \leftrightarrow \mathbf{V}^{\langle {\bf 2} , N_{\bf 2}\rangle} \vDash   \psi(\hat{x_1},\cdots,\hat{x_n})$. Moreover, for any restricted negation-free formula $\phi$, we have $\phi(x_1,\cdots,x_n) \leftrightarrow \mathbf{V}^{\langle {\bf A}, N\rangle} \vDash   \phi(\hat{x_1},\cdots,\hat{x_n})$. 

\end{itemize}

\end{lem}

The proof of the last theorem is the same for the intuitionistic case because we consider restricted negation-free formulas and it will be useful to prove the validity of axiom Infinity.

\subsection{ Validating axioms}

Now, we will  prove the validity of all set-theoretic axiom from $ZFC_\omega$. First,  let us start by considering a fixed model $\mathbf{V}^{\langle {\bf A}, N\rangle}$ over a complete  $C_\omega$-structure $\langle {\bf A}, N\rangle$ and suppose that $\mathbf{V}^{\langle {\bf A}, N\rangle}$ verifies Leibniz law; i.e., it is a Leibniz model. For the sake of brevity and from now on, we will use the notation $||.||$ instead of $||.||^{\langle A, N\rangle}$ without any risk of confusion. Then:

 \
 
\noindent {\bf Extensionality}

Given $x,y\in \mathbf{V}^{\langle {\bf A}, N\rangle}$, then

\begin{eqnarray*}
  ||\forall z (z\in x \leftrightarrow z\in y)|| &= & ||\forall z ((z\in x \to z\in y) \wedge (z\in y \to z\in x) ||\\
&= & \bigwedge\limits_{z\in \mathbf{V}^{\langle A, N\rangle}} (||z\in x|| \to ||z\in y||)\wedge \bigwedge\limits_{z\in \mathbf{V}^{\langle A, N\rangle}} (||z\in y|| \to ||z\in x||) \\
&\leq  & \bigwedge\limits_{z\in dom(x)} (||z\in x|| \to ||z\in y||)\wedge \bigwedge\limits_{z\in dom(y)} (||z\in y|| \to ||z\in x||) \\
&\leq & \bigwedge\limits_{z\in dom(x)} ( x(z) \to ||z\in y||)\wedge \bigwedge\limits_{z\in dom(y)} ( y(z) \to ||z\in x||)  \\
 &= & ||x=y||. \\
\end{eqnarray*}

Thus, we have $||\forall x \forall y\forall z ((z\in x \leftrightarrow z\in y)\to (x=y)) ||$. On the other hand, for any $z\in \mathbf{V}^{\langle A, N\rangle}$ we infer that $||x=y|| \wedge ||z\in x||\leq ||z\in y||$ and so, $||x=y|| \leq ||z\in x||\to ||z\in y||  $. Therefore, $||\forall x \forall y ((x=y) \to \forall z (z\in x \leftrightarrow z\in y)) ||$.

\

\noindent {\bf Pairing} 

Let $u,v\in \mathbf{V}^{\langle {\bf A}, N\rangle}$ and consider the function $w=\{\langle u,1\rangle, \langle v,1\rangle \}$. Thus, we have that $|| z\in w||= (w(u)\wedge ||z=u||) \vee (w(v)\wedge ||z=v||)=  ||z=u||\vee ||z=v||= ||z=u\vee z=v||$.

\

\noindent {\bf Powerset}

Assume $u\in \mathbf{V}^{\langle {\bf A}, N\rangle}$ and suppose $w$ is a function such that $dom(w)=\{f:dom(u)\to A: f\, \hbox{function}\}$ and $w(x)=||\forall y\in x(y\in u)||$. Therefore, 

$$|| v\in w||=\bigvee\limits_{x\in dom(w)} (||\forall y\in x(y\in u)|| \wedge || x=v||)\leq ||\forall y\in v(y\in u)|| .$$

On the other hand,  given $v\in \mathbf{V}^{\langle {\bf A}, N\rangle}$ and considering the function $a$ such that $dom(a)=dom(u)$ and $a(z)=||z\in u|| \wedge ||z\in v ||$. So, it is clear that $a(z)\to ||z\in v ||=1$ for every $z\in dom(a)$, therefore

\begin{eqnarray*}
  ||\forall y\in v(y\in u)||  &= & \bigwedge\limits_{y\in dom(v)} (v(y)\to  || y\in u||)\\
&= &  \bigwedge\limits_{y\in dom(v)} (v(y)\to  (|| y\in u|| \wedge v(y))) \\
&\leq  & \bigwedge\limits_{y\in dom(v)} (v(y)\to  a(y)) \\
&\leq & \bigwedge\limits_{y\in dom(v)} ( v(y) \to ||y\in a||)\wedge \bigwedge\limits_{z\in dom(a)} ( a(z) \to ||z\in v||)  \\
 &= & ||v=a||. \\
\end{eqnarray*}
Since $a(y)\leq ||y\in u||$ for every $y\in dom(a)$, then we have $||\forall y\in a(y\in u)||=1$. Now by construction we have that $a\in dom(w)$ and so,  $ ||\forall y\in v(y\in u)||\leq ||\forall y\in a(y\in u)|| \wedge ||v=a||= w(a) \wedge ||v=a||\leq ||v\in w||$.

\

\noindent {\bf Union} 

Given $u\in \mathbf{V}^{\langle {\bf A}, N\rangle}$ and considering the function $w$ with $dom(w)= \bigcup\limits_{v\in dom(u)} dom(v)$ and $w(x)= \bigvee\limits_{v\in A_x} v(x)$ where $A_x=\{v\in dom(u): x\in dom(v)\}$. Then,

\begin{eqnarray*}
 ||y\in w||  &= &\bigvee\limits_{x\in dom(w)} (||x \approx y|| \wedge \bigvee\limits_{v\in A_x} v(x))\\
&= & \bigvee\limits_{x\in dom(w)}  \bigvee\limits_{v\in A_x} (||x\approx y|| \wedge v(x) )  \\
&= & \bigvee\limits_{v\in dom(u)}  \bigvee\limits_{x\in dom(v)}(||x \approx y|| \wedge v(x) ) \\
&= & || \exists v\in u(y\in v)||.\\
\end{eqnarray*}

\noindent {\bf Separation}

Given $u\in \mathbf{V}^{\langle {\bf A}, N\rangle}$ and supposing $dom(w)=dom(u)$ and $w(x)=||x\in u||\wedge ||\phi(x)||$ then

\begin{eqnarray*}
  || z\in w|| &= & \bigvee\limits_{x\in dom(w)} (||y\in w||\wedge ||\phi(y)||\wedge ||y \approx z||)\\
&\leq & \bigvee\limits_{x\in dom(w)} (||\phi(z)||\wedge ||y \approx z||).  \\
\end{eqnarray*}

Besides, 

\begin{eqnarray*}
 ||\phi(z)||\wedge ||y \approx z|| &= &\bigvee\limits_{y\in dom(u)} (u(y)\wedge ||z \approx y|| \wedge ||\phi(z)||)\\
&\leq & \bigvee\limits_{y\in dom(u)} (||y\in u||\wedge ||z \approx y|| \wedge ||\phi(y)||)  \\
&= & \bigvee\limits_{y\in dom(u)} (w(y) \wedge ||z\approx y||) = ||z\in w||. \\
\end{eqnarray*}

\noindent {\bf Empty set} 

Note that $||u=u||=1$ for all $u\in \mathbf{V}^{\langle {\bf A}, N\rangle}$ and then, $||\neg(u\approx u)||\in N_1$. Therefore, let us consider a function $w\in \mathbf{V}^{\langle {\bf A}, N\rangle}$ such that  $u\in dom(w)$ and $ran(w)\subseteq \{||\neg(u\approx u)||\}$, then it is clear that $||u\in w|| = \bigvee\limits_{x\in dom(w)} (w(x)\wedge ||u\approx x||)=||\neg(u\approx u)||$.

\

\noindent {\bf Infinity} 

Assume the formula $\psi(x)$ is ``$\emptyset\in x \wedge \forall y\in x (y^+\in x)$''. Then, the axiom in question is the sentence $\exists x \psi(x)$. Now, it is clear that the negation-free formula $\emptyset\in x \wedge \forall y\in x (y^+\in x)$ is restricted and certainly $\psi(\omega)$ true. Hence, by Lemma \ref{tecnico} (iv), we get $||\psi(\hat{\omega})||=1$, and so, $||\exists x \psi(x)||=1$.

\

\noindent {\bf  Collection} 

Given $u\in \mathbf{V}^{\langle {\bf A}, N\rangle}$ and $x\in dom(u)$, there exists by Axiom of Choice some ordinal $\alpha_x$ such that $\bigvee\limits_{y \in \mathbf{V}^{\langle {\bf A}, N\rangle}} ||\phi(x,y)||= \bigvee\limits_{y \in \mathbf{V}^{\langle {\bf A}, N\rangle}_{\alpha_x}} ||\phi(x,y)||$. For $\alpha=\{\alpha_x: x\in dom(u)\}$ and $v$ the function with domain  $\mathbf{V}^{\langle {\bf A}, N\rangle}$ and range $\{1\}$, we have:

\begin{eqnarray*}
  ||\forall x\in u\exists y \phi(x,y)  ||  &= & \bigwedge\limits_{x\in dom(u)} (u(x)\to \bigvee\limits_{y\in \mathbf{V}^{\langle A, N\rangle}} ||\phi(x,y)  ||) \\
&= & \bigwedge\limits_{x\in dom(u)} (u(x)\to \bigvee\limits_{y\in \mathbf{V}^{\langle A, N\rangle}_\alpha} ||\phi(x,y)  ||) \\
&=  & \bigwedge\limits_{x\in dom(u)}(u(x) \to ||\exists y\in v \phi(x,y)  ||) \\
 \\
 &= & ||\forall x\in u \exists y\in v \phi(x,y) ||    \\
 &\leq & ||\exists w \forall x\in u \exists y\in w \phi(x,y) ||. \\
\end{eqnarray*}

\noindent {\bf Induction}

Let us suppose $x\in \mathbf{V}^{\langle {\bf A}, N\rangle}$, then we will prove its  validity  by induction on the well-founded relation $y\in dom(x)$. Now, assume that $a=||\forall x\big[(\forall y\in x\psi(y))\to \psi(x)\big]||$.

On the other hand, assume that $a\leq ||\psi(y)||$ for every $y\in dom(x)$.  So, it is clear that $a\leq \bigwedge\limits_{y\in dom(x)} ||\psi(y)|| \leq \bigwedge\limits_{y\in dom(x)} \big(x(y)\wedge ||\psi(y)||\big) = || \forall y \in x \psi(y)||$. But $a\leq ||(\forall y\in x\psi(y))|| \to ||\psi(x)||$. Therefore, $a\leq \big(||(\forall y\in x\psi(y))|| \to ||\psi(x)||\big)\wedge || \forall y \in x \psi(y)||\leq || \psi(x)||$ as required.

\

From the above, we have proved the following Theorem:

\begin{theo}
Let $\langle {\bf A}, N\rangle$ be a  complete $C_\omega$-structure such that $\mathbf{V}^{\langle {\bf A}, N\rangle}$ verifies the Leibniz law (i.e. it is a Leibniz model). Then, the all set-theoretic axioms of $ZFC_\omega$  are valid in $\mathbf{V}^{\langle {\bf A}, N\rangle}$.
\end{theo}

Now, we define the paraconsistent set theories $ZFC_{\omega,n}$ with ($n<\omega$) ( and   $ZFC_{\omega,\infty}$) from of the first order  version  of $C_{\omega,n}$ (and $C_{\omega,\infty}$, see Section \ref{preli}) over the first-order signature ${\Theta}_\omega$ which contains an equality predicate\, $\approx$ \, and a binary predicate $\in$  by adding the set-theoretic axiom schemas from Definition \ref{IZF}. Then, we have also  proved:

\begin{coro}
Let $\langle {\bf A}, N\rangle$ be a complete $C_{\omega,n}$($C_{\omega,\infty}$)-structure  such that $\mathbf{V}^{\langle {\bf A}, N\rangle}$ verifies the Leibniz law (i.e. it is a Leibniz model). Then, the all set-theoretic axioms of $ZFC_{\omega,n}$ ($ZFC_{\omega,\infty}$)  are valid in $\mathbf{V}^{\langle {\bf A}, N\rangle}$.
\end{coro}


On the other hand, we have that for our paraconsistent set theories the axiom of scheme {\bf Comprehension} is not valid in our models. To prove this, it is enough to see that  $||\exists x \forall y(y\in x)||=0$, this formula is an instance of {\bf Comprehension}, see  \cite[Theorem 6.3]{BT}. 



\end{document}